\documentclass[12pt]{article}
\usepackage{amsmath,amsthm,amssymb,amsfonts,color}

\usepackage[latin1]{inputenc}
\usepackage{cite}
\usepackage[english]{babel}

\newcommand{\R}{{\mathbb R}}

\newcommand{\Z}{{\mathbb Z}}          

\newcommand{\uni}{{\mathfrak u}}     %

\newcommand{\symp}{{\mathfrak{sp}}}
\newcommand{\XIS}{{\mathfrak X}}

\newcommand{\rr}{\rightarrow}
\newcommand{\lrr}{\longrightarrow}

\newcommand{\na}{{\nabla}}
\newcommand{\nab}[2]{{\nabla_{#1}#2}}
\newcommand{\Tr}[1]{{\mathrm{Tr}}\,#1}

\newcommand{\End}[1]{{\mathrm{End}}\,#1}

\newcommand{\Id}{{\mathrm{Id}}}
\newcommand{\grad}{{\mathrm{grad}}\,}
\newcommand{\dx}{{\mathrm{d}}}
\newcommand{\inv}[1]{{#1}^{-1}}
\newcommand{\cyclic}{\mathop{\kern0.9ex{{+}
\kern-2.2ex\raise-.28ex\hbox{\Large\hbox{$\circlearrowright$}}}}\limits}

\newcommand{\papa}[2]{\frac{\partial#1}{\partial#2}}

\newcommand{\cinf}[1]{{\mathrm{C}}^\infty_{#1}}

\newtheorem{teo}{Theorem}[section]

\newtheorem{coro}{Corollary}[section]
\newtheorem{prop}{Proposition}[section]

\numberwithin{equation}{section}

\begin{document}


\thispagestyle{empty}

\begin{center}

\markright{\sl\hfill  Albuquerque --- Picken \hfill}

{\huge{\bf{On invariants of 

\vspace{.8cm}

almost symplectic connections}}}

\vspace{1.2cm}

R. Albuquerque\footnote{Departamento de Matem\'atica da Universidade de \'Evora and Centro de Investiga\c c\~ao em Matem\'atica e Aplica\c c\~oes (CIMA-U\'E), Rua Rom\~ao Ramalho, 59, 7000-671 \'Evora, Portugal.}
\hspace{5cm}
R. Picken\footnote{Departamento de Matem\'atica and CAMGSD - Centro de
An\'alise Matem\'atica, Geometria e Sistemas Din\^amicos, Instituto Superior
T\'ecnico, Technical University of Lisbon, Avenida Rovisco Pais, 1049-001 Lisboa, Portugal.}

\ \ \ rpa@uevora.pt  \hspace{3.9cm}  rpicken@math.ist.utl.pt

\vspace{1cm}

\date{\today}
\end{center}


{\bf Abstract:} We study the irreducible decomposition under $Sp(2n,\R)$ of the space of torsion tensors of almost symplectic connections. Then a description of all symplectic quadratic invariants of torsion-like tensors is given. When applied to a manifold $M$ with an almost symplectic structure, these instruments give preliminary insight for finding a preferred linear almost symplectic connection on $M$. We rediscover Ph.~Tondeur's Theorem on almost symplectic connections. Properties of torsion of the vectorial kind are deduced.

\vspace*{4mm}


\vspace*{3mm}


\vspace{4mm}

Authors partially supported by Funda\c{c}\~{a}o Ci\^{e}ncia e Tecnologia, Portugal, through their research centres.

\vspace{4mm}

\section{Torsion tensors and symplectic quadratic invariants}
\label{Ttatsqi}

\subsection{Introduction}
\label{Introduction}

In \cite{BourgCahen} a variational principle was discovered over the space of
symplectic connections of a given symplectic manifold. It enables us to select
a preferred type of connection amongst those important instruments of symplectic
geometry. The Euler-Lagrange equations, or field equations, were deduced in the
same article and properties of the preferred connections were studied
subsequently by others (\cite{Rawnsley1,Rawnsley2,Rawnsley3,Vais1,Vais2}).

Already in \cite{Vais1} we find the decomposition into irreducible parts of the
curvature tensor of symplectic connections, which satisfies the Bianchi identity
due to the torsion-free condition. There are two types of symplectic curvature
tensors, which were given the names of Ricci and Weyl. The \textit{preferred}
or so-called Ricci-type connections correspond to the vanishing of the Weyl
part of the curvature.

It was proved, cf. \cite{Vais2,Rawnsley1,Rawnsley2}, that the Ricci type
connections solve the integrability equations of the respective twistor space of
the symplectic manifold, i.e. the bundle of linear tangent complex structures
endowed with a canonical complex structure induced by the symplectic connection.
Further insight into twistor theory was given in \cite{AlbuRawnsley}, sustaining
the idea that more interactions with Hamiltonian mechanics and complex geometry
should be pursued. Hence our present and hopefully future interest in the
subject.

In this article we wish to explore the almost symplectic case, i.e. we consider
a smooth manifold $M$ endowed with a non-degenerate 2-form $\omega$ and a linear
connection $\na$ such that $\na\omega=0$. Recall the vanishing of the torsion
yields $\omega$ closed. Since the Lie algebra $\symp(2n,\R)$ identifies with
$S^2(\R^{2n})$, the space of almost symplectic connections is in 1-1
correspondence with the space of sections of 2-symmetric tensors with values on
$T^*M$.

We prove a few original results following our study of almost symplectic
structures. We rediscover one remarkable result of Ph.~Tondeur, Theorem
\ref{tequalsdw0}, which seems to have been left aside through time. Our main
focus is on the torsion tensor, its $Sp(2n,\R)$-irreducible components and the
study of its scalar invariants. We deduce there is essentially one quadratic
invariant.

The subject of almost symplectic structures appeared in foundational works such
as those of the mathematicians H.-C.~Lee, P.~Libermann, A.~Lichnerowicz and
Ph.~Tondeur. Many new developments have emerged on symplectic Yang-Mills
theories, symplectic Dirac operators and the geometry of Fedosov manifolds. We
refer the interested reader to recent \cite{Hab2Rosen,GelRetShubin,Uraka}.

\subsection{Representation of tensors under $Sp(2n,\R)$}
\label{Representation of tensors}

Let $(V,\omega)$ be a real symplectic vector space of dimension $m=2n$ and let $G$ be the symplectic group $Sp(2n,\R)$. Let $S^kV$ and $\Lambda^kV$ denote respectively the space of symmetric and skew-symmetric multivectors. Notice $S^0V=\Lambda^0V=\R$ and $S^1V=\Lambda^1V=V$.

We wish to study here the decomposition into irreducibles under the group $G$ of
the space $V\otimes\Lambda^2V$, which, as is well known, agrees with the fibre
of the vector bundle of torsion tensors of a linear connection over an
$m$-dimensional manifold.

We start by recalling an exact sequence, due to Koszul:
\begin{equation}\label{exaseq1}
0\rr S^lV\stackrel{A_{l,1}}\lrr S^{l-1}V\otimes V\rr\ldots
S^{l-k}V\otimes\Lambda^{k-1}V\stackrel{A_{l-k,k}}\lrr S^{l-k-1}V\otimes\Lambda^kV\ldots\rr\Lambda^lV\rr0
\end{equation}
where $A_{p,q}:S^pV\otimes \Lambda^{q-1}V\rr S^{p-1}V\otimes\Lambda^qV$ is given
by
\begin{equation}\label{exaseq2}
A_{p,q}(u_1\cdots u_p\otimes v_1\wedge\cdots\wedge v_{q-1})=
\sum_{i=1}^pu_1\cdots \widehat{u_i}\cdots u_p\otimes v_1\wedge\cdots\wedge v_{q-1}\wedge u_i
\end{equation}
and $A_{0,q}=0$. It is trivial to see $A_{p-1,q+1}A_{p,q}=0$. To see that the
kernel of $A_{p-1,q+1}$ is in the image of $A_{p,q}$ we recall a dual exact
sequence of maps $B_{p,q}:S^{p-1}V\otimes \Lambda^qV\rr
S^pV\otimes\Lambda^{q-1}V$ defined by
\[  B_{p,q}(u_1\cdots u_{p-1}\otimes v_1\wedge\cdots\wedge v_q)=
\sum_{i=1}^q(-1)^iv_iu_1\cdots u_{p-1}\otimes v_1\wedge\cdots\wedge\widehat{v_i}\wedge\cdots\wedge v_q  \]
and $B_{p,0}=0$. We have $A_{p+1,q}\circ B_{p+1,q}-B_{p,q+1}\circ
A_{p,q+1}=(-1)^{q}(p+q)\Id$ on $S^p\otimes\Lambda^q$ (so we conclude $B$ is a
right inverse on the kernel of $A$ and reciprocally), thus proving exactness of
(\ref{exaseq1}). Recall that any isomorphism $g\in GL(V)$ acts by
$g\cdot(u_1\cdots u_p\otimes v_1\wedge\cdots\wedge v_q)=gu_1\cdots gu_p\otimes
gv_1\wedge\cdots\wedge gv_q$, inducing natural representation spaces
$S^pV\otimes\Lambda^qV$. Then clearly the maps $A$ and $B$ above are
$GL(V)$-homomorphisms or equivariant.

Now we concentrate on the short exact sequence
\begin{equation}\label{exaseq3}
0\lrr S^3V\stackrel{A_{1}}\lrr S^{2}V\otimes V\stackrel{A_{2}}\lrr
V\otimes\Lambda^2V\stackrel{A_3}\lrr\Lambda^3V\lrr0 
\end{equation}
where $A_i=A_{4-i,i}$, just to notice that we must study $A_2$ carefully. From now on we consider just the symplectic group and so the first thing to notice is that the isomorphism $V\cong V^*,\ v\mapsto v^*=\omega(v,\ )$ is a $G$-morphism. Let us define also a map $\varphi$ by
\begin{equation}\label{defidevarfi}
\begin{split}
S^3V\stackrel{A_{1}}\lrr S^{2}V\otimes V\stackrel{\varphi}\lrr V\lrr0\hspace{3cm}\\
\varphi(u_1u_2\otimes v)=\omega(u_1,v)u_2+\omega(u_2,v)u_1.
\end{split}
\end{equation}
It is clearly well defined and a $G$-epimorphism. Then
$A_1(S^3V)\subset\ker\varphi$ because
\begin{eqnarray*}
\varphi(A_1(u_1u_2u_3))
&=&\cyclic_{i\in\Z_3}\omega(u_i,u_{i+2})u_{i+1}+\omega(u_{i+2},u_{i})u_{i+1}=0 .
\end{eqnarray*}
There is a right inverse for $\varphi$ defined through each symplectic basis
$e_1,\ldots,e_n,e_{n+1},\ldots,e_{2n}$ (we
assume $\omega(e_i,e_j)=0,\ \omega(e_i,e_{j+n})=\delta_{ij},\
\omega(e_{i+n},e_{j+n})=0$ for $i,j=1,\ldots,n$). Let us fix one of these bases
and let $\xi$ be given by
\begin{equation}\label{quesi}
\xi(v)=\frac{1}{2n+1}\sum_{i=1}^ne_iv\otimes e_{i+n}-e_{i+n}v\otimes e_i.
\end{equation}
\begin{prop} 
The map $\xi:V\rr S^2V\otimes V$ is a $G$-monomorphism, does not depend on the basis and satisfies $\varphi\xi=\Id$.
\end{prop}
\begin{proof} We notice that $\xi(v)=\frac{-1}{2n+1}B_{2,2}(v\otimes\omega)$, so
the first property follows. Let $g\in G$. Since $g\omega=\omega=e_i\wedge 
e_{i+n}$ (we use Einstein's convention on sums from 1 to $n$ here and from now
on), $\xi$ is $G$-equivariant and does not depend on the basis. Finally, let
$v=v_ie_i+v_{i+n}e_{i+n}$. Then
\begin{eqnarray*}
(2n+1)\varphi(\xi(v)) & = & \omega(e_i,e_{i+n})v+\omega(v,e_{i+n})e_i-\omega(e_{i+n},e_i)v-\omega(v,e_i)e_{i+n}\\
& = & 2nv+v_ie_i+v_{i+n}e_{i+n}\ =\ (2n+1)v
\end{eqnarray*}
as required.
\end{proof}
Let $\pi=-\tfrac{1}{3}B_{3,1}:S^{2}V\otimes V\rr S^3V$ be the left inverse of
$A_1$. Hence $\pi$ is defined by $\pi(u_1u_2\otimes v)=\tfrac{1}{3}u_1u_2v$. One
immediately checks that $\pi A_1=\Id$ and $\pi(\xi(v))=0$. Now let $\eta$ be the
$G$-endomorphism of $S^2V\otimes V$ defined by $\eta=\Id-A_1\pi-\xi\varphi$,
hence given by
\begin{equation}\label{defideeta}
\begin{split}
\eta(u_1u_2\otimes v)\ =\ u_1u_2\otimes v-A_1(\pi(u_1u_2\otimes v))-\xi(\varphi(u_1u_2\otimes v)) \ \ \ \ \ \ \ \ \\  =\ \tfrac{2}{3}u_1u_2\otimes v-\tfrac{1}{3}u_2v\otimes u_1-\tfrac{1}{3}u_1v\otimes u_2-\xi(\varphi(u_1u_2\otimes v)).
\end{split}
\end{equation}
\begin{teo}\label{teo1}
We have that $\mathrm{Im}\,\eta=\ker\varphi\cap\ker\pi={\cal A'}$ and
\begin{eqnarray}\label{decirredut}
S^2V\otimes V=S^3V\oplus{\cal A'}\oplus V
\end{eqnarray}
is the decomposition into $G$-irreducible subspaces. The dimension of $\cal A'$ is $\tfrac{8}{3}(n^3-n)$.
\end{teo}
\begin{proof}
Since $V\subset\ker\pi$ via $\xi$, we find that $\pi\eta=\pi-\pi A_1\pi=0$.
Furthermore, we proved $\mathrm{Im}\,A_1\subset\ker\varphi$. Hence
$\varphi\eta=\varphi-\varphi\xi\varphi=0$. Moreover $\eta$ is the identity in
$\cal A'$. It is well known that all powers $S^kV$ are $G$-irreducible.
Now for the irreducibility of $\cal A'$ we appeal to a result of J.~Rawnsley, {\em Notes on
$W=0$} (unpublished). The methods are also explained in \cite{BursRawn}, as the reader may see,
relying on a classical Theorem for the decomposition of tensor products of irreducible
representations such as $S^2V$ and $V$. The highest weights are known for each factor, the
Theorem says the highest weights correspond to irreducibles and gives an immediate algorithm on
how to get them for the tensor product.

Recall that for any $m$-dimensional vector space $V$, we have 
$\dim S^kV={{m+k-1}\choose{k}}$. Therefore
\begin{equation*}
\dim{\cal A'}  =  \frac{m(m+1)m}{2}-\frac{(m+2)(m+1)m}{6}-m \ = \ \frac{m^3-4m}{3}
\end{equation*}
and the formula for $m=2n$ follows.
\end{proof}
There is another natural endomorphism of $S^2V\otimes V$ which is important to be aware of. It is defined by $\chi(u_1u_2\otimes v)=vu_2\otimes u_1+vu_1\otimes u_2$. One easily checks that $\chi^2-\chi-2\Id=0$ and that $\ker\pi$ and $S^3V$ are respectively the $-1$ and $2$ eigenspaces of $\chi$.

\subsection{Torsion tensors}
\label{Torsion tensors}

Now we continue with the study of the torsion-like tensors. Consider the map $C:V\otimes \Lambda^2V\rr V$ given by $C(v\otimes u\wedge w)=\omega(u,w)v+\omega(v,u)w-\omega(v,w)u$
--- which is clearly well defined.
\begin{teo}\label{teo2}
The space of torsion-like tensors has the decomposition into $G$-irreducible subspaces
\begin{eqnarray}\label{torirredut}
\Lambda^2V\otimes V\simeq{\cal A'}\oplus V\oplus{\cal T'}\oplus V
\end{eqnarray}
where ${\cal T'}=\ker C\cap \Lambda^3V$. Then ${\cal T'}$ satisfies ${\cal T'}\oplus V=\Lambda^3V$.
\end{teo}
\begin{proof}
The decomposition follows from sequence (\ref{exaseq3}) and the previous
Theorem. We must check $A_2(S^2V\otimes V)=A_2({\cal A'}\oplus V)=\ker A_3$ is
contained in $\ker C$. This is true since
\begin{eqnarray*}
\lefteqn{C(v\otimes w\wedge u+u\otimes w\wedge  v)\ =} \\ & & \omega(w,u)v+\omega(v,w)u+\omega(u,v)w+\omega(w,v)u+\omega(v,u)w+\omega(u,w)v\ =\ 0.
\end{eqnarray*}
Each $v\in V$ appears on the right hand side, i.e. inside $\Lambda^3V$, as
$v\otimes \omega=v\otimes e_i\wedge e_{i+n}$ because then $A_3(v)=\omega\wedge
v$. Therefore 
\[ C(v)=C(v\otimes\omega)= \omega(e_i,e_{i+n})v+\omega(v,e_i)e_{i+n}-\omega(v,e_{i+n})e_i=(n-1)v \]
and so an element $f\in\Lambda^3V$ verifies $C^2f=(n-1)Cf$. Moreover, it
decomposes as $f-\frac{1}{n-1}Cf+\frac{1}{n-1}Cf$ according to ${\cal T'}\oplus
V$. The irreducibility of $\cal T'$, the primitive 3-forms, is well known and
part of the Hodge-Lepage sequence. Its dimension is
${{2n}\choose{3}}-2n=\frac{2}{3}n(2n^2-3n-2)$.
\end{proof}
A torsion-like tensor in any of the invariant subspaces $V$ is said to be of vectorial type. We shall see their natural formulation in section \ref{Ascwvt}.

The above result is very close to the orthogonal group decomposition of metric
torsion tensors found by \'E.~Cartan, as one may see in \cite{Agricola1}, due to
the fact that $U(n)=G\cap O(2n)$. Further obstacles in other geometrically
relevant representation spaces are to be found also in \cite{HansProc,RanNeri}.
According to \cite{Vais1,Weyl}, an approach to the irreducibility of $\cal A,T'$
relies on the quadratic invariants which we study below.

\subsubsection{Curvature tensors}
\label{Curvature tensors}

The curvature-type tensors we shall use later live in the space $\Re^G=S^2V\otimes\Lambda^2V$. To study them we have the exact sequence
\begin{equation}\label{exaseqcurva}
0\lrr S^4V\stackrel{A_{1}}\lrr S^{3}V\otimes V\stackrel{A_{2}}\lrr
S^2V\otimes\Lambda^2V\stackrel{A_3}\lrr V\otimes\Lambda^3V\stackrel{A_4}\lrr \Lambda^4V\lrr0 
\end{equation}
where $A_i=A_{5-i,i}$. Curvature tensors may or may not satisfy the Bianchi
identity. This corresponds exactly with the kernel of $A_3$; indeed it is easy
to see that the Bianchi map is $\Id_V\otimes A_3$ with $A_3$ of
(\ref{exaseq3}) composed with the inclusion
$S^2V\otimes\Lambda^2V\hookrightarrow V\otimes V\otimes\Lambda^2V$, and this
composition is the map $A_3$ of (\ref{exaseqcurva}). 

We know from \cite{Vais1} that $\ker A_3={\cal E}\oplus{\cal W}$, irreducibly.
The contraction of $V$ with $S^3V$ using $\omega$ yields one $S^2V$, which
corresponds to the so-called Ricci-type curvatures (of Ricci-type connections,
for which a formula is found eg. in \cite{BourgCahen}).


\subsection{Symplectic quadratic invariants}
\label{Symplectic invariants}

Let us continue with the symplectic vector space $(V,\omega)$ of dimension
$m=2n$. Associated to the space of $k$-tensors $Q\in\otimes^kV$ we have
polynomial symplectic invariants in the components of $Q$. We refer to the
theory explained in \cite{Vais1}. By \cite{Weyl} it is known that every
polynomial invariant is algebraically generated by so called $h$-products of
symplectic traces, i.e. degree $h$ homogeneous polynomials 
\begin{equation}\label{hpst}
r(Q)=\sum_{i_1,\ldots,i_{2t}=1}^{2n}\hat{\sigma}(\otimes^hQ)_{i_1,\ldots,i_{2t}}\omega^{i_1i_2}\cdots\omega^{i_{2t-1}i_{2t}},
\end{equation}
where $kh=2t$ and $\sigma\in S_{2t}$ is some fixed permutation acting on the
indices of $\otimes^hQ$. As usual, $\omega^{ij}$ denotes the inverse matrix of
$\omega(e_i,e_j)$ given from some chosen basis $e_i, i=1,\ldots,2n$. Indeed,
since for $g\in G$ the induced transformations are $e_\alpha=g_{\alpha i}e_i$
and $\omega_{\alpha\beta}=g_{\alpha i}g_{\beta j}\omega_{ij}$, implying
$\omega^{\alpha\beta}=g^{i'\alpha}g^{j'\beta}\omega^{i'j'}$, the $G$-invariance
holds.

Now we concentrate on a 3-tensor $Q$ and the quadratic invariants. Note that for
odd tensors we can only consider even degree $h$-products of symplectic traces.
The symbol ``$...$'' will always denote $Q_{ijk}Q_{pql}$. We also use the
Einstein summation convention throughout.
\begin{teo}\label{qsi}
i) If $Q$ is symmetric in any two indices, then every quadratic symplectic
invariant on $Q$ is $0$.
\\
ii) If $n>1$, the quadratic symplectic traces are classified by the following list:
\begin{equation}\label{class}
\begin{split}
r_1(Q)=...\,\omega^{ik}\omega^{jp}\omega^{ql}\ \ \ \ \ \ \
r_2(Q)=...\,\omega^{ij}\omega^{kp}\omega^{ql} \\
r_3(Q)=...\,\omega^{ik}\omega^{jl}\omega^{pq}\ \ \ \ \ \ \ r_4(Q)=...\,\omega^{iq}\omega^{jl}\omega^{kp} 
\end{split}
\end{equation}
which are linearly independent symplectic invariants.
\\
iii) If $n=1$, then in (\ref{class}) we have $r_1=r_2=-r_3=r_4\ne0$. If moreover $Q$ is skew-symmetric in any two indices, then every quadratic symplectic invariant is 0.
\\
iv) Up to scalar factor there is a unique non-vanishing quadratic symplectic trace:
\begin{equation}\label{ssqst}
r(Q)=Q_{ijk}Q_{pql}\omega^{ij}\omega^{kp}\omega^{ql}.
\end{equation} 
on the space of tensors $Q$ such that $Q_{ijk}=-Q_{jik}$.\\
v) If $Q_{ijk}$ is totally skew-symmetric, then $r(Q)=0$.
\end{teo}
\begin{proof}
As we observed previously, our study is restricted to quadratic symplectic
traces. Let $Q_{ijk}$ denote any 3-tensor. It is easy to understand that is
enough to analyse the following traces (note the $r_1,r_2,r_3,r_4$ below are not
those in the Theorem). With $i,j,k$ in a first fixed position:
\begin{eqnarray*}
r_1Q=...\,\omega^{ik}\omega^{jq}\omega^{pl}\quad  &\  
r_2Q=...\,\omega^{ik}\omega^{jp}\omega^{ql}\quad &\
r_3Q=...\,\omega^{ik}\omega^{jl}\omega^{pq}
\end{eqnarray*}
and then in a second fixed position:
\begin{eqnarray*}
   r_4Q=...\,\omega^{ip}\omega^{jl}\omega^{kq}\ &\
r_6Q=...\,\omega^{iq}\omega^{jp}\omega^{kl}\ &\
r_8Q=...\,\omega^{il}\omega^{jp}\omega^{kq} \\
 r_5Q=...\,\omega^{ip}\omega^{jq}\omega^{kl}\ &\
r_7Q=...\,\omega^{iq}\omega^{jl}\omega^{kp} \ &\
r_9Q=...\,\omega^{il}\omega^{jq}\omega^{kp}
\end{eqnarray*}
(applying $\hat{\sigma}$ on the, of course, equivalent to (\ref{hpst}) with the different
permutations). Now in cases $a=1,4,5,6,9$, the invariant $r_aQ$ is $0$ because of the labelling
permutation $i\leftrightarrow p,\ j\leftrightarrow q,\ k\leftrightarrow l$. For instance, in
case 1 we find
\[ r_1Q=Q_{pql}Q_{ijk}\omega^{pl}\omega^{qj}\omega^{ik}=-r_1Q\: =\,0 \]
and in case 4 the permutation yields
\[ r_4Q=Q_{pql}Q_{ijk}\,\omega^{pi}\omega^{qk}\omega^{lj}=-r_4Q\:=\,0.\]
Cases 5, 6 and 9 are analogous. By the same permutation,
$r_{8}=...\,\omega^{pk}\omega^{qi}\omega^{lj}=-r_7$ and, finally,
$r_{9}=...\,\omega^{pk}\omega^{qj}\omega^{li}=-r_9=0$.

Notice we have other cases to consider, formally analogous to the above $r_a,\,a=1,2,3$: 
\begin{eqnarray*}
r_{1'}Q=...\,\omega^{ij}\omega^{kq}\omega^{pl} \   &\   
r_{2'}Q=...\,\omega^{ij}\omega^{kp}\omega^{ql} \   &\
r_{3'}Q=...\,\omega^{ij}\omega^{kl}\omega^{pq}      \\
&\   r_{2''}Q=...\,\omega^{jk}\omega^{ip}\omega^{ql} \   &
\end{eqnarray*}
and the reader may further see a $r_{1''}=-r_2$ and a $r_{3''}=-r_{2'}$. Apart from $r'_2$, all
these are only formally new, since we have $r_{1'}=...\,\omega^{pq}\omega^{lj}\omega^{ik}=-r_3$,
clearly $r_{3'}=0$ and $r_{2''}=...\,\omega^{ql}\omega^{pi}\omega^{jk}=-r_{2''}=0$.

For the case of $r_2$ with $Q_{pql}$ symmetric in $ql$ or in $pl$, then clearly we get
$r_2Q=...\,\omega^{ik}\omega^{jp}\omega^{ql}=0$. If it is symmetric in $pq$, then $r_2Q$ becomes
\[ Q_{jik}Q_{pql}\omega^{jk}\omega^{ip}\omega^{ql}= Q_{qpl}Q_{jik}\omega^{ql}\omega^{pi}\omega^{jk}=0.   \]
In case of $r_7$, with $Q_{pql}$ symmetric in $ql$, then we have 
\[ r_7Q=...\,\omega^{il}\omega^{jq}\omega^{kp}=Q_{pql}Q_{ijk}\omega^{pj}\omega^{qk}\omega^{li}= ...\,\omega^{pk}\omega^{qj}\omega^{li}=0.   \]
The same is true for symmetries in $pq$ or $pl$, as straightforward computations would show.
Since $r_3$ and $r'_2$ are similar to $r_2$ we have proved {\it i)}.

Finally we must prove linear independence in general of $r_2,r_{2'},r_3,r_7$, which agree with
the
$r_1,r_2,r_3,r_4$ in (\ref{class}). This is a simple task running through four examples, easy to
find, which make all invariants vanish except one. For instance, if we take $Q=1_{113}+1_{324}$,
then $r_2=1,\ r_{2'}=0,\ r_3=0,\ r_7=0$ using a symplectic basis --- for which the first four
vectors satisfy $\omega^{1,2}=\omega^{1,4}=\omega^{2,3}=\omega^{3,4}=0,\
\omega^{1,3}=\omega^{2,4}=1$. With $r_{2'}$ we take $Q=1_{113}+1_{324}$, with $r_3$ we take
$Q=1_{113}+1_{243}$ and with $r_7$ we take $Q=1_{122}+1_{434}$.

Notice the case $n=1$ is different. We use a symplectic basis, such that $\omega^{12}=-\omega^{21}=1$. Then we compute:
\begin{eqnarray*}
r_2Q & = &  Q_{112}Q_{212}-Q_{112}Q_{221}-Q_{122}Q_{112}+Q_{122}Q_{121} \\
& &  -Q_{211}Q_{212}+Q_{211}Q_{221}+Q_{221}Q_{112}-Q_{221}Q_{121}
\end{eqnarray*}
\begin{eqnarray*}
r_{2'}Q & = &  Q_{121}Q_{212}-Q_{122}Q_{112}+Q_{122}Q_{121}-Q_{121}Q_{221} \\
& &  -Q_{211}Q_{212}+Q_{212}Q_{112}-Q_{212}Q_{121}+Q_{211}Q_{221}\ =\ r_2Q
\end{eqnarray*}
\begin{eqnarray*}
r_3Q & = &  Q_{112}Q_{122}-Q_{122}Q_{121}-Q_{112}Q_{212}+Q_{122}Q_{211} \\
& &  -Q_{211}Q_{122}+Q_{221}Q_{121}+Q_{211}Q_{212}-Q_{221}Q_{211}\ =\ -r_2Q
\end{eqnarray*}
\begin{eqnarray*}
r_7Q & = &  Q_{111}Q_{222}-Q_{112}Q_{122}-Q_{121}Q_{221}-Q_{211}Q_{212} \\
& &  +Q_{221}Q_{211}+Q_{212}Q_{112}+Q_{122}Q_{121}-Q_{222}Q_{111}\ =\ r_2Q
\end{eqnarray*}
and this is non zero as the example $Q=1_{112}+1_{122}$ shows. This proves {\it ii)} and {\it
iii)}.

To prove {\it iv)} suppose $Q$ is such that $Q_{ijk}=-Q_{jik}$. Then it is easy to see 
\[ r_2Q=Q_{jik}Q_{qpl}\omega^{jk}\omega^{iq}\omega^{pl} = Q_{pql}Q_{ijk}\omega^{pl}\omega^{qi}\omega^{kj}=-r_2Q=0\] 
and in the same way $r_{2'}=r_3$ and $r_7=r_6=0$. We should also verify that $r_{2'}$ does not
vanish: let $Q$ be given by
\begin{equation*}
 Q_{ijk}=\left\{\begin{array}{lcl}
 1 & & (i,j,k)=(1,3,2)\ \mbox{or}\ (4,3,1)\\
 -1 & & (i,j,k)=(3,1,2)\ \mbox{or}\ (3,4,1)\\
0 & & \mbox{elsewhere}
\end{array}\right. ,
\end{equation*}
which in the previous notation would be $Q=1_{132}-1_{312}+1_{431}-1_{341}$. Then $Q$ is
skew in $ij$ and $r_{2'}Q=2Q_{132}Q_{pql}\omega^{13}\omega^{2p}\omega^{ql}=2$.

Finally, to prove {\it v)} we see that the hypotheses on $Q$ implies $r_{2'}Q=r_{1}Q=0$.
\end{proof}
We may state, directly from case {\it iv)} above, the following.
\begin{coro}\label{uniquesympinv}
If $Q_{ijk}=Q_{ij}^h\omega_{hk}$ is skew-symmetric in $ij$, then the space of quadratic symplectic invariants is generated by $Q_{ij}^pQ_{pq}^q\omega^{ij}$.
\end{coro}
\begin{proof} 
Indeed, $r(Q) = Q_{ij}^hQ_{pq}^o\omega_{hk}\omega_{ol}\omega^{ij}\omega^{kp}\omega^{ql}
         =  Q_{ij}^pQ_{qp}^q\omega^{ij}$.
\end{proof}

\section{Almost symplectic connections}

\subsection{The Theorem of Tondeur}

It is well known that a manifold $M$ admits a non-degenerate 2-form $\omega\in\Omega^2$ if and only if $M$ is almost complex. So we will be considering the category of almost complex manifolds, but from the perspective of symplectic geometry. We only assume $M$ has a preferred non-degenerate 2-form $\omega$.

In the theory of linear connections one may consider almost symplectic connections
$\nabla:\Omega^0(TM)\rr\Omega^1(TM)$, i.e. such that  $\nab{ }{\omega}=0$. If the torsion
\begin{equation}
T^\nabla(X,Y)=\nab{X}{Y}-\nab{Y}{X}-[X,Y]
\end{equation}
is 0, $\forall X,Y\in\XIS_M=\Omega^0(TM)$, then the connection is called \textit{symplectic}.
\begin{prop}[cf. \cite{Vais1,BourgCahen}]
Let $M$ admit a non-degenerate 2-form $\omega$.\\
i) There exists an almost symplectic connection $\nabla$ on $M$.\\
ii) The space of all almost symplectic connections with fixed torsion is $\Omega^0(S^3T^*M)$.
\end{prop}
\begin{proof}
Let $\tilde{\nabla}$ be any connection on $M$. Let $\nabla=\tilde{\nabla}+A$, with $A\in \Omega^1(\End{TM})$ given by $\omega(A_XY,Z) =\frac{1}{2}\tilde{\nabla}_X\omega(Y,Z)$. Then
\begin{eqnarray*}
\nab{X}{\omega}(Y,Z)=\tilde{\nabla}_X\omega(Y,Z)-\omega(A_XY,Z)-\omega(Y,A_XZ)=0
\end{eqnarray*}
as required. If $\nabla'-\nabla=A$ is the difference between two connections such that
$\nabla'\omega=\nab{}{\omega}=0$, then we may say $A$ belongs to $\Omega^0(S^2T^*M\otimes
T^*M)=\Omega^1(S^2T^*M)$. If moreover $A_XY=A_YX$, the condition for equal torsion, then we have
that $A$ is a completely symmetric 3-tensor.
\end{proof}
It is quite frequently stated that a symplectic manifold admits a symplectic connection. We
were pleased to be able to generalise this result to make it into an equivalence statement, but
it turned out the result was already known, due to Philippe Tondeur, cf.
\cite{GelRetShubin,Tondeur}. We present it here with a simple original proof.
\begin{teo}[Ph. Tondeur]\label{tequalsdw0}
There exists an almost symplectic connection $\nabla$ on $M$ such that, $\forall X,Y\in\XIS_M$,
\begin{equation}\label{tequalsdw}
\omega(T^\na(X,Y),Z)=\frac{1}{3}\dx\omega(X,Y,Z).
\end{equation}
A symplectic connection exists on $M$ if and only if $\dx\omega=0$.
\end{teo}
\begin{proof}
On any manifold we always have a torsion free linear connection $\nabla^0$. Suppose $\nabla=\nabla^0+A$ is an almost symplectic connection, given by the above Proposition,
and let a 1-form $B\in\Omega^1(\End{TM})$ be defined by
\begin{equation}\label{defiB}
 \omega(B_XY,Z)=a[\omega(A_YX,Z)+\omega(A_ZX,Y)]+b[\omega(A_YZ,X)+
\omega(A_ZY,X)] 
\end{equation}
with $a,b$ to be determined. Since (\ref{defiB}) is symmetric in $Y,Z$, the resulting connection $\nabla^1=\nabla+B$ is almost symplectic. Since $T^{\nabla^1}(X,Y)=T^{\nabla^0}(X,Y)+A_XY-A_YX+B_XY-B_YX$, we find
\begin{eqnarray*}
\omega(T^{\nabla^1}(X,Y),Z)& =& (1-a)\omega(A_XY-A_YX,Z)+(a-b)\omega(A_ZX,Y)\\
   & & -b\omega(A_XZ,Y)+(a-b)\omega(-A_ZY,X)+b\omega(A_YZ,X).
\end{eqnarray*}
Choosing $a,b$ such that $a-b=b$ and $1-a=b$, that is, $a=2/3$ and $b=1/3$, yields
\begin{equation*}
  \omega(T^{\nabla^1}(X,Y),Z)=   \frac{1}{3}\bigl(\omega(T^{\nabla}(X,Y),Z)+\omega(T^{\nabla}(Y,Z),X)+\omega(T^{\nabla}(Z,X),Y)\bigl)
\end{equation*}
which is equal to $\frac{1}{3}\dx\omega(X,Y,Z)$, due to the well known formula
\begin{equation*}
\dx\omega(X,Y,Z)=\cyclic_{X,Y,Z}\bigl(\nabla_X\omega\,(Y,Z)+\omega(T^{\nabla}(X,Y),Z)\bigr).
\end{equation*}

The last formula shows, reciprocally, that any $\nabla$ symplectic implies $\omega$ closed, which completes our proof.
\end{proof}
Suppose we are in the Hermitian setting $(M,J,g,\omega)$, with $g$ the metric,
$\omega=J\lrcorner g$ and $J$ compatible. Then we have an almost symplectic connection: the
canonical Hermitian connection $\na_X=D_X-\frac{1}{2}J(D_XJ)$ induced from the Levi-Civita
connection $D$ of $M$. Indeed $\na$ is a $\uni(n)$-connection. Then the equation given by
Tondeur's Theorem $T^\na=\frac{1}{3}\dx\omega$ is verified by the nearly K\"ahler
structures, i.e. those $J$ satisfying
\begin{equation}
 g((D_XJ)X,Y)=0,\ \forall X,Y.
\end{equation}
Notice by Gray-Hervella's classification of almost Hermitian geometries
through the triple $J,\omega,D$, the nearly K\"ahler condition is equivalently given as
$D\omega=\frac{1}{3}\dx\omega$.

\subsection{Symplectic invariants of torsion tensors}

Let $M$ be an almost symplectic $2n$-manifold, distinguished by a non-degenerate 2-form. Let $\na$ be any almost symplectic connection and $T^\na$ denote its torsion. Let $e_1,\ldots,e_{2n}$ denote any local frame on $M$, let $\omega_{ij}=\omega(e_i,e_j)$ and $\omega^{pq}$ denote the inverse: $\omega_{ij}\omega^{jq}=\delta_i^q$. Furthermore, let 
\begin{equation}\label{Tijk}
T_{ijk}=\omega(T^\na(e_i,e_j),e_k)=T_{ij}^h\omega_{hk}.
\end{equation}
By Theorem \ref{qsi}, in particular by (\ref{ssqst}), we may define the invariant
\begin{equation}\label{invariante}
t^\na= T_{ijk}T_{pql}\omega^{ij}\omega^{kp}\omega^{ql}.
\end{equation}
Of course, the summation indices $i,j,k,p,q,l$ vary from 1 to $2n$.
\vspace{2mm}\\
\textsc{Remark.}
By Corollary \ref{uniquesympinv} we also have $t^\na\ =\  T_{ij}^pT_{qp}^q\omega^{ij}$.
And $t^\na=-\rho^\na_{ij}\omega^{ij}$, where $\rho^\na(X,Y)=\Tr{\,T^\na(T^\na(X,Y),\cdot)}$, which is a 2-form on $M$. This is a straightforward computation. Other possible traces will only give the trivial invariant.
\begin{teo}
Up to a scalar multiple, $t^\na$ is the only real quadratic $Sp(2n,\R)$-invariant on the torsion of linear connections on $M$. 

If $\omega(T^\na(X,Y),Z)$ is totally skew-symmetric or $T^\na$ is in the space of torsion tensors corresponding to $\cal A'$ in Theorem \ref{teo2}, then $t^\na=0$.
\end{teo}
\begin{proof}
The first part is straightforward from Corollary \ref{uniquesympinv}. The second follows from part {\it v)} of Theorem \ref{qsi}. For the third statement notice $\cal A'$ can only be in the kernel of the trace map $T\mapsto T_{pql}\omega^{ql}$. Then recall the definition of $t^\na$.
\end{proof}
\begin{prop}
If $M$ is a symplectic manifold, then $t^\na=0$ for any almost symplectic connection $\na$.
\end{prop}
\begin{proof}
By Theorem \ref{tequalsdw0} we may assume the existence of a symplectic connection. The
difference between this and any given almost symplectic $\na$ is a tensor $A$ such that
$\omega(A_XY,Z)$ is symmetric in $Y,Z$. Since $T^\na(X,Y)=A_XY-A_YX$, we find
\begin{equation*}
 t^\na=(A_{ijk}-A_{jik})(A_{pql}-A_{qpl})\omega^{ij}\omega^{kp}\omega^{ql}=0
\end{equation*}
by part \textit{i} of Theorem \ref{qsi}.
\end{proof}
\textsc{Remark.} Given an almost symplectic $\na$ and a direction $A\in\Omega^1(\symp(TM,\omega))$, we have a ray of almost symplectic connections $\na^s=\na+sA,\ s\in\R$, and then we easily deduce its variation:
\begin{equation}\label{derivativeofinvariantt}
 {\papa{ }{s}}_{|s=0}t^{\na^s}=\bigl(2A_{ijk}T_{pql}-A_{qpl}T_{ijk}\bigr)\omega^{ij}\omega^{kp}\omega^{ql}.
\end{equation}
This corresponds to $\dx t^\na(A)$, the derivative of the quadratic invariant at the point
$\na$ in the space of connections. We lack a characterization of the critical points of $t^\na$,
i.e. we do not know what the vanishing of (\ref{derivativeofinvariantt}) for all $A$ tells
us about $\na$. 

We may also consider gauge transformations: we choose any $g\in\Omega^0(Sp(TM,\omega))$ and
produce $\na'_XY=g\circ\na_X \circ\inv{g}\ Y=\na_XY-(\na_Xg)\inv{g}Y$. The direction $A=-(\na
g)\inv{g}$ is not stable under the quadratic invariant, nor is its derivative.

Now let $\phi:M\lrr M$ be a symplectomorphism on the manifold $M$ endowed with a linear
connection. We shall restrict our attention to the case of an almost symplectic connection
$\na$, although this is not essential. Furthermore, what follows may be said in the wider
context of symplectomorphisms between two almost complex manifolds with chosen $\omega$ and
$\na$. The reader can easily adapt what follows to the latter framework.

Recall that any diffeomorphism $\phi$ acts on vector and tensor fields. It also acts on
connections, 
\begin{equation}
 (\phi\cdot\na)_XY=\phi\cdot\bigl(\na_{\inv{\phi}\cdot X}\inv{\phi}\cdot Y\bigr)
\end{equation}
$\forall X,Y\in\XIS_M$. The functoriality of this action is well known: $T^{\phi\cdot\na}=\phi\cdot T^\na$, and analogously with the curvature tensor.
\begin{prop}
The group $\mathrm{Symp}(M,\omega)$ acting on the space of almost symplectic connections
transforms the invariant $t^\na$ by
\begin{equation}\label{tphina}
 t^{\phi\cdot\na}=t^\na\circ\inv{\phi}.
\end{equation}
Moreover, if $M$ is connected, $\int_M t^\na\frac{\omega^n}{n!}$ is an invariant of the orbit
of $\na$.
\end{prop}
\begin{proof}
 Since $\phi^*\omega=\omega$, we also have $\phi\cdot\na\:\omega=0$. Now, as $(\inv{\phi}\cdot e_i)_x=\dx\inv{\phi}_y(e_i),\ \forall x\in M$, where $y=\phi(x)$ and $e_i$ is any basis, we have
\begin{equation*}
 \omega_y(T^{\phi\cdot\na}(e_i,e_j),e_k)= \omega_x(T^{\na}(\inv{\phi}\cdot e_i,\inv{\phi}\cdot e_j),\inv{\phi}\cdot e_k).
\end{equation*}
We used the invariance of $\omega$ and $\inv{\omega}$ under $\phi$. 
Hence we get the stated formula and 
\begin{equation*}
 \int_Mt^{\phi\cdot\na}\,\frac{\omega^n}{n!}=\int_M\phi^*\bigl(t^{\phi\cdot\na}\,\frac{\omega^n}{n!}\bigr)=\int_Mt^{\na}\,\frac{\omega^n}{n!}
\end{equation*}
since $\phi$ preserves orientation.
\end{proof}

\subsection{Almost symplectic connections with vectorial torsion}
\label{Ascwvt}

As we saw earlier, almost symplectic connections of vectorial torsion can be of two types. We now give some results on the quadratic invariant for these.

In \cite{Nannicini2} we find the notion of ``conformal class of an almost symplectic manifold $(M,\omega)$ with a fixed almost symplectic connection $\na$'':
\begin{equation}
 {\cal C}=\{ (e^{2f}\omega,\na^f):\  f\in \cinf{M}(\R)\}
\end{equation}
where 
\begin{equation}
 \na^f_XY=\na_XY+X(f)Y+Y(f)X+\omega(X,Y)\grad f
\end{equation}
and $\grad f$ is the symplectic gradient: $\omega(\grad f, X)=X(f)$. 

Since $(\na^f)^g=\na^{f+g}$ we deduce the transitivity of the conformal factor; the class arises from an equivalence relation.

It is proved in \cite{Nannicini2} that $\na^f(e^{2f}\omega)=0$ and $T^{\na^f}=T^\na+2\omega\,\grad f$. We add the following remark.
\begin{prop}
 Writing $\na^f=\na+A^f$, we have a $\R$-linear map $A:\cinf{M}\rr$\-$\Omega^1(\End{TM})$ such that $A^{f_1f_2}=f_1A^{f_2}+f_2A^{f_1}$. Two connections $\na,\na^f$ share the same unparametrized geodesics on the level sets of $f$.
\end{prop}

Now we shall consider the case when $\omega$ is closed, and hence we may assume that $\na$ is
torsion-free. We then fix a function $f$ and introduce another vector field $U\in\XIS_M$. We
shall study the more general case of $\na^{U,f}$ given by
\begin{equation}
 \omega(\na^{U,f}_XY,Z)=\omega(\na^f_XY,Z)+\frac{1}{2}\bigl(\omega(X,Y)\omega(U,Z)+\omega(U,Y)\omega(X,Z)\bigr).
\end{equation}
Notice the symmetry in $Y,Z$ on the rhs, implying that $\na^{U,f}$ is still an almost symplectic
connection for $\omega'=e^{2f}\omega$.
\begin{prop}
Let $\na$ be a torsion-free, symplectic connection on $M,\omega$ and let $U\in\XIS_M$ be any
vector field. Then the torsion of $\na^{U,f}$ is given by
\begin{equation}\label{tvectorialsimples}
T^{\na^{U,f}}(X,Y) = \omega(X,Y)(2\grad(f)+U)+\frac{1}{2}\omega(U,Y)X-\frac{1}{2}\omega(U,X)Y
\end{equation}
and hence $t^{\na^{U,f}}=2e^{-2f}(2n^2-n-1)U(f)$.
\end{prop}
\begin{proof}
The first formula is trivial. For the second, consider the following tensor
\begin{equation}\label{vectorialtorsionagain}
 T(X,Y)=\omega(X,Y)A+\omega(X,W)Y-\omega(Y,W)X
\end{equation}
with $A,W$ fixed. Then a long but simple computation yields
\begin{equation*}
 T_{ijk}T_{pql}\omega^{ij}\omega^{kp}\omega^{ql}=2(2n^2-n-1)\omega(A,W).
\end{equation*}
The result now follows introducing $A=U+2\grad f,\ W=\frac{1}{2}U$, which is the case for the torsion of $\na^{U,f}$, finding
\begin{equation*}
 t^{\na^{U,f}}=2e^{-2f}(2n^2-n-1)\omega(A,W)=2e^{-2f}(2n^2-n-1)U(f).
\end{equation*}
The conformal factor is indeed $e^{(2-3)2f}$.
\end{proof}
The torsion in (\ref{tvectorialsimples}) is of vectorial type of the most general kind, as we deduce from (\ref{torirredut}) in section \ref{Torsion tensors}.
Notice $W=-A$ corresponds with the totally skew symmetric case, cf. (\ref{vectorialtorsionagain}).

\vspace{2cm}

\bibliographystyle{plain}

\begin{thebibliography}{10}



\bibitem{Agricola1}
I. Agricola and Chr. Thier, 
{\em The geodesics of metric connections with vectorial torsion},
Ann. Glob. Anal. Geom. 26 (2004), 321--332.


\bibitem{AlbuRawnsley}
R.~Albuquerque and J.~Rawnsley,
\newblock {\em Twistor Theory of Symplectic Manifolds},
\newblock J. Geom. Phys. 56 (2006), 214--246.



\bibitem{BourgCahen}
F.~Bourgeois and M.~Cahen,
\newblock {\em A variational principle for symplectic connections},
\newblock J. Geom. Phys., Vol 30, Issue 3 (1999) 233--265.

\bibitem{BursRawn}
F. Burstall and J. Rawnsley,
{\em Affine connections with $W=0$}, (2007)\\
{ http://arxiv.org/abs/math/0702032 },

\bibitem{Cabrera1}
F.~Cabrera,
\newblock {\em Special almost Hermitian geometry},
\newblock J. Geom. Phys., Vol 55, Issue 4 (2005), 450--470.


\bibitem{Rawnsley2}
M.~Cahen, S.~Gutt, J.~Horowitz and J.~Rawnsley,
\newblock {\em Homogeneous symplectic manifolds with ricci-type curvature},
\newblock J. Geom. Phys., Vol 38, Issue 2 (2001), 140--151.

\bibitem{Rawnsley3}
M.~Cahen, S.~Gutt, J.~Horowitz and J.~Rawnsley,
\newblock {\em Moduli space of symplectic connections of ricci type on $t^{2n}$; a formal approach},
\newblock J. Geom. Phys., Vol 46, (2003), 174--192.


\bibitem{Rawnsley1}
M.~Cahen, S.~Gutt and J.~Rawnsley,
\newblock {\em Symmetric symplectic spaces with Ricci-type curvature},
\newblock Kluwer Acad., Math. Phys. Stud., 22, 2000,
\newblock Conf\'erence Mosh\'e Flato 1999, Vol. II (Dijon).


\bibitem{GelRetShubin}
I. Gelfand, V. Retahk and M. Shubin,
{\em Fedosov manifolds},
Adv. Math. 136, (1998) 104--140.


\bibitem{Hab2Rosen}
K. Habermann, L. Habermann and P. Rosenthal,
{\em Symplectic Yang--Mills theory, Ricci tensor, and connections},
Calc. of Variations and Part. Dif. Eq., Vol 30, 2 (2007).




\bibitem{Koba}
S.~Kobayashi,
\newblock {\em Differential geometry of complex vector bundles},
\newblock Iwanami Shoten, Princeton University Press, 1987.

\bibitem{HansProc}
H. Kraft and C. Procesi,
{\em Classical Invariant Theory A Primer},
Preliminary version, July 1996.


\bibitem{Dusa}
D.~McDuff and D.~Salamon,
\newblock {\em Introduction to symplectic topology},
\newblock Oxford Math. Monographs. Oxford Uni. Press, New York,
1998.



\bibitem{Nannicini1}
A.~Nannicini.
\newblock {\em Twistor bundles of almost symplectic manifolds},
\newblock Rend. Istit. Mat. Univ. Trieste, 30 (1998), 91--106.

\bibitem{Nannicini2}
A.~Nannicini,
\newblock {\em Twistor methods in conformal almost symplectic geometry},
\newblock Rend. Istit. Mat. Univ. Trieste, 34 (2002), 215--234.

\bibitem{RanNeri}
G. Rangarajan and F. Neri, 
{\em Canonical representations of $sp(2n,\R)$},
J. Math. Phys., Vol. 33, No. 4 (1992).



\bibitem{Tondeur}
Ph.~Tondeur,
{\em Affine Zusammenh\"ange auf Mannigfaltigkeiten mit
fast\-sym\-plek\-ti\-scher Struk\-tur},
Comment. Math. Helv. 36 (1962), 234-244. 

\bibitem{Uraka}
H. Urakawa,
{\em Yang--Mills Theory over Compact Symplectic Manifolds},
Ann. Global Anal. Geom., Vol 25, 4 ( 2004), 365--402.



\bibitem{Vais1}
I.~Vaisman,
\newblock {\em Symplectic curvature tensors},
\newblock Monatsh. Math., 100(4) (1985), 299--327.

\bibitem{Vais2}
I.~Vaisman,
\newblock {\em Symplectic twistor spaces},
\newblock J. Geom. Phys., 3(4) (1986), 507--524.

\bibitem{Weyl}
H.~Weyl,
\newblock {\em Classical Groups, Their Invariants and Representations},
\newblock Princeton University Press, 1946.


\end{thebibliography}


\end{document}